\documentclass[ 12pt,reqno]{amsart} 
\usepackage{amssymb,amsmath,amsthm,hyperref} 
\usepackage{euscript}
\usepackage{a4wide}
\usepackage[T2A]{fontenc}
\usepackage[utf8]{inputenc}
\usepackage{amsfonts}
\usepackage{amssymb, amsthm}
\usepackage{amsmath}
\usepackage{cite}
\usepackage{graphicx}
\usepackage{tikz}
\usetikzlibrary{calc}
\usetikzlibrary{decorations.pathmorphing}
\theoremstyle{plain}
\newtheorem{thm}{Theorem}
\newtheorem{prop}{Proposition}

\newtheorem{remark}{Remark}

\newtheorem{conj}{Conjecture}

\renewcommand{\C}{\mathbb{C}}

\newcommand{\RK}{\rotatebox[origin=c]{180}{RK}}
\newcommand{\KR}{\operatorname{KR}}
\renewcommand{\phi}{\varphi}
\numberwithin{equation}{section} 
\numberwithin{thm}{section} 
\numberwithin{prop}{section} 
\numberwithin{defn}{section} 
\numberwithin{remark}{section}

\title{Further $q$-reflections on the modulo $9$ Kanade--Russell (conjectural) identities}

 \author{Stepan Konenkov}
 \address{Department of Mathematics and Computer Science, Saint Petersburg State University, St. Petersburg 199178, Russia}
 \email{konenkov.stepan@yandex.ru.}

\begin{document} 
\date{\today}



\begin{abstract}
We examine four identities conjectured by Dean Hickerson which complement five modulo $9$ Kanade--Russell identities, and we build up a profile of new identities and new conjectures similar to those found by Ali Uncu and Wadim Zudilin.
\end{abstract}

\maketitle
\pagestyle{plain}
\section{Introduction}
This work takes place in a relatively new and rapidly developing area of algebra, number theory, and combinatorics, the Shashank Kanade and Matthew Russell conjectural identities \cite{KR1, M1}, which turn out to be double-sum Rogers--Ramanujan type identities after work of Ka\u{g}an Kur\c{s}ung\"oz \cite{Ku}.   For some recent applications in algebra, see for example \cite{PSW, Tsuch}.  Our study is motivated by Ali Uncu and Wadim Zudilin's recent work on the Kanade--Russell identities\cite{UZ}, and Dean Hickerson's subsequent conjectures  related to work of Uncu and Zudilin \cite{Hick}.  In this paper we will develop a profile for Hickerson's conjectures in the context of work of Uncu and Zudilin on the Kanade--Russell identities.

We begin with the partition-theoretic formulations of the first four Kanade--Russell conjectures \cite[$(I_1)$-$(I_4)$]{KR1}, \cite{M1}.  For example $(I_1)$ becomes 

\smallskip
{\em The number of partitions of a non-negative integer n into parts congruent to $1$, $3$, $6$, or $8$ $\pmod 9$ is the same as the number of partitions of $n$ with difference at least $3$ at distance $2$ such that if two consecutive parts differ by at most $1$, then their sum is divisible by $3$.}

\smallskip
\noindent The other formulations \cite[$(I_2)$-$(I_4)$]{KR1}, \cite{M1} are similar.   Each formulation is the conjectured equality of two generating functions.  For each conjecture, it is easy to see that one of the two generating functions is a simple infinite product.  Subsequently, Kur\c{s}ung\"oz used keen partition methods to show that each of the complementary generating functions is in fact a double-sum, thus giving Andrews--Gordon type series \cite{Ku}.  For example, $(I_1)$ reads
\begin{equation*}
\sum\limits_{m,n\geq 0} \frac{q^{m^2+3mn+3n^2}}{(q;q)_m (q^3;q^3)_n}
\stackrel{?}{=}\frac{1}{(q,q^3,q^6,q^8;q^9)_{\infty}}.
\end{equation*}

Building on work of Kur\c{s}ung\"oz \cite{Ku}, Uncu and Zudilin took an experimental approach to understanding the world of Kanade--Russell conjectures by considering the $q\to q^{-1}$ reflections of the unproved identities.  This approach is motivated by George Andrews's duality theory for Rogers--Ramanujan type identities \cite{AND2} with respect to Rodney Baxter's solution to the hard hexagon model of statistical mechanics \cite{Baxter}.   In physics, $q\to q^{-1}$ sometimes has meaning in terms of flows; in other words, you are converging to different conformal field theories whose representation theory is described by characters given by limits of the polynomials being considered.  One can also use the $q\to q^{-1}$ as a guide to go from mock theta function identities to partial theta function identities \cite{Mort}.

Andrews \cite{AND2} found finite-sum, i.e. polynomial, $q$-series identities which approach Rogers--Ramanujan identities in the limit.  If one replaces $q$ with $q^{-1}$ in the finite-sum identities, multiplies by a high enough power of $q$, and lets the number of terms go to infinity, then one obtains new Rogers--Ramanujan type identities.  Uncu and Zudilin started off with five Kanade--Russell identities, found finite versions, studied their reflections, and built up a profile of new and conjectural identities.   Even more recently, Hickerson \cite{Hick} conjectured four more double-sum Rogers--Ramanujan type identities which complement the original five studied by Uncu and Zudilin.

We briefly review the duality theory of Andrews \cite{AND2} for Rogers--Ramanujan type identities with respect to Baxter's solution to the hard hexagon model of statistical mechanics \cite{Baxter}. Although the $q \rightarrow q^{-1}$ transform makes sense for a $q$-hypergeometric series, it does not for a product. In a duality theory initiated by Andrews \cite{AND1} for various sets of identities of Rogers--Ramanujan type, Andrews used finite versions, i.e. polynomial or rational function identities, which converge to infinite $q$-series in the limit.  To state an example, we first introduce some standard notation.

We define $q$-Pochhammer notation and the Gaussian binomial coefficient \cite{AND3}.  Let $|q|<1$ be a nonzero complex number.  We define
\begin{gather*}
(x)_n=(x;q)_n:=\prod_{i=0}^{n-1}(1-q^ix), \ \ (x)_{\infty}=(x;q)_{\infty}:=\prod_{i\ge 0}(1-q^ix).
\end{gather*} 
Let $m,n\in\mathbb{Z}$. The Gaussian binomial coefficient is then defined by the product 
\begin{equation*}
 \begin{bmatrix} m\\ n\end{bmatrix}_q:=
 \begin{cases} \frac{(q;q)_m}{(q;q)_n(q;q)_{m-n}}
 & \textup{if} \ m \ge n \ge 0,\\
 0&  \textup{if} \ n<0 \ \textup{or} \ m<n.
 \end{cases}
 \end{equation*}

To translate between identities for Regions I and IV from \cite{AND2, Baxter}, Andrews used identities which have origins in work of Schur \cite{AND1,Schur}, 
\cite[4.1, 4.2]{AND2}:
\begin{align}
\sum_{j=0}^{\infty}q^{j^2}\begin{bmatrix} N-j \\ j   \end{bmatrix}_q
&=\sum_{\lambda=-\infty}^{\infty}(-1)^{\lambda}q^{\lambda(5\lambda+1)/2}\begin{bmatrix} N \\ \lfloor \frac{N-5\lambda}{2}   \rfloor  \end{bmatrix}_q,\label{equation:Andrews-4.1}\\
\sum_{j=0}^{\infty}q^{j^2+j}\begin{bmatrix} N-j \\ j   \end{bmatrix}_q
&=\sum_{\lambda=-\infty}^{\infty}(-1)^{\lambda}q^{\lambda(5\lambda-3)/2}\begin{bmatrix} N+1 \\ \lfloor  \frac{N+1-5\lambda}{2}   \rfloor  +1 \end{bmatrix}_q,\label{equation:Andrews-4.2}
\end{align}
where $\lfloor x \rfloor$ is the greatest integer less than or equal to $x$. As $N\rightarrow \infty$, we can apply the Jacobi Triple Product Identity to the right-hand side of each of the two equations, and they become the Rogers--Ramanujan identities, i.e. for Region I:  
\begin{align*}
\sum_{n=0}^{\infty}\frac{q^{n^2}}{(q)_n}&=\frac{1}{(q;q^5)_{\infty}(q^4;q^5)_{\infty}},\\
\sum_{n=0}^{\infty}\frac{q^{n^2+n}}{(q)_n}&=\frac{1}{(q^2;q^5)_{\infty}(q^3;q^5)_{\infty}}.
\end{align*}
To obtain the identities for Region IV, replace $N$ with $2N+a$, $a\in\{0,1\}$,  then replace $q$ with $q^{-1}$, multiply by a power of $q$ to clear any negative exponents of $q$, and then let $N\rightarrow \infty$.  Identity (\ref{equation:Andrews-4.1}) implies
{\allowdisplaybreaks \begin{align}
\sum_{n=0}^{\infty}\frac{q^{n^2}}{(q)_{2n}}&=\frac{1}{(q;q^2)_{\infty}(q^4,q^{16};q^{20})_{\infty}},
\label{equation:RRtype1}\\
\sum_{n=0}^{\infty}\frac{q^{n^2+n}}{(q;q)_{2n+1}}&=\frac{1}{(q,q^2,q^8,q^9;q^{10})_{\infty}(q^5,q^6,q^{14},q^{15};q^{20})_{\infty}}.\label{equation:RRtype2}
\end{align}}%
We point out that (\ref{equation:RRtype1}) is found in \cite[p. 331]{Rog} and that (\ref{equation:RRtype2}) is implied by equations $2$ and $3$ in \cite[p. 330]{Rog}.  Identity (\ref{equation:RRtype2}) is also found in \cite[(94)]{Slat}

Uncu and Zudilin examined the five modulo 9 Kanade--Russell identities  \cite[(11)-(15)]{UZ} through the $q\to q^{-1}$ reflection.  They took an approach similar to Andrews's.  However, after applying $q\rightarrow q^{-1}$ to a given finite version of one of the double-sum Rogers--Ramanujan type identities, they obtain three new conjectural identities instead of two.  This is because one must now consider the finite parameter $N$ modulo $3.$  In the profiles developed by Uncu and Zudilin they thus obtain fifteen new conjectural identities; however, they were always able to prove one in each group of three.

Recently, Hickerson discovered four new conjectural double-sum identities, which complete the original list of five.  As we will see in Section \ref{section:Hickerson}, Hickerson \cite{Hick} showed that the $\C(q)$-span of a certain generic building block is $9$-dimensional over $\mathbb{C}(q)$, with five independent elements given by Kanade and Russell.  The remaining four were found by Hickerson.

In this work, we will study the $q$-reflections of Hickerson's identities and build up a family of new identities and conjectural identities analogous to work of Uncu and Zudilin.  We will begin with Hickerson's four conjectural double-sum Rogers--Ramanujan type identities.  In the course of this step, we make modifications to partition theoretic arguments that  Uncu and Zudilin used to find finite polynomial versions of the Kanade--Russell identities.  We point out that Uncu and Zudilin's partition-theoretic methods \cite{UZ} build on those of Kur\c{s}ung\"oz \cite{Ku, Ku2}.  We then find finite polynomial versions of Hickerson's double-sums, see Propositions \ref{theorem:MainFiniteList}, \ref{theorem:SecondFiniteList}.

Once we have finite polynomial versions of the double-sums, we make the $q\to q^{-1}$ reflection and proceed to determine the double-sum sides of what one hopes to be new identities and conjectures.  We show that each of Hickerson's original four double-sums lead to three new double-sums, see Theorems \ref{theorem:S11-limits}, \ref{theorem:S24-limits}, \ref{theorem:S0-1-limits}, and \ref{theorem:S02-limits}.  Our situation is slightly different than that of Uncu and Zudilin, so it is not clear how many new double-sums Rogers--Ramanujan type identities we should have.  Using $q$-series techniques and properties of $q$-binomial coefficients, we determine and prove four new identities, see Theorem \ref{stats}.  We are also able to find closed forms of two conjectural identities, see Conjectures \ref{conj:Conjecture-1} and \ref{conj:Conjecture-2}.  However, taking conjugates of the two new conjectural identities yields an additional two conjectures, which brings us to four new conjectural identities.

In contrast, the five identities studied by Uncu and Zudilin led to five new and ten conjectural identities.  Here we appear to be missing some conjectural identities, but that will be the subject of future work.

\section{Motivation and known results}

We review some basic facts and refer the reader to \cite{AND3} for details on the $q$-Pochhammer symbol and $q$-binomial coefficients.  We also recall useful properties of the $q$-binomial coefficient.

\begin{thm}\cite[Theorem $3.2$]{AND3} \label{thm: prop_bin_coeff}
Let $0\le m\le n$ be integers.  The Gaussian polynomial $\begin{bmatrix}n\\m\end{bmatrix}_q$ is a polynomial of degree $m(n-m)$ satisfying
\begin{enumerate}
\item $\begin{bmatrix}n\\0\end{bmatrix}_q=\begin{bmatrix}n\\n\end{bmatrix}_q=1$,
\item $\begin{bmatrix}n\\m\end{bmatrix}_q=\begin{bmatrix}n\\n-m\end{bmatrix}_q$,
\item $\begin{bmatrix}n\\m\end{bmatrix}_q=\begin{bmatrix}n-1\\m\end{bmatrix}_q
+q^{n-m}\begin{bmatrix}n-1\\m-1\end{bmatrix}_q$ for $n\ge1$,
\item $\begin{bmatrix}n\\m\end{bmatrix}_q=\begin{bmatrix}n-1\\m-1\end{bmatrix}_q
+q^{m}\begin{bmatrix}n-1\\m\end{bmatrix}_q$ for $n\ge1$.
\end{enumerate}
\end{thm}

We recall the double-sum form of five partition-theoretic identities discovered by Kanade and Russell \cite[$(I_1)$-$(I_4)$]{KR1}, \cite{M1}:
{\allowdisplaybreaks \begin{gather}
    \KR_1(q):=\sum\limits_{m,n\geq 0} \frac{q^{m^2+3mn+3n^2}}{(q;q)_m (q^3;q^3)_n}\stackrel{?}{=}\frac{1}{(q,q^3,q^6,q^8;q^9)_{\infty}}, \label{KR1}\\
    \KR_2(q):=\sum\limits_{m,n\geq 0} \frac{q^{m^2+3mn+3n^2+m+3n}}{(q;q)_m (q^3;q^3)_n}\stackrel{?}{=}\frac{1}{(q^2,q^3,q^6,q^7;q^9)_{\infty}}, \label{KR2}\\
    \KR_3(q):=\sum\limits_{m,n\geq 0} \frac{q^{m^2+3mn+3n^2+2m+3n}}{(q;q)_m (q^3;q^3)_n}\stackrel{?}{=}\frac{1}{(q^3,q^4,q^5,q^6;q^9)_{\infty}}, \label{KR3}\\
    \KR_4(q):=\sum\limits_{m,n\geq 0} \frac{q^{m^2+3mn+3n^2+m+2n}}{(q;q)_m (q^3;q^3)_n}\stackrel{?}{=}\frac{1}{(q^2,q^3,q^5,q^8;q^9)_{\infty}}, \label{KR4}\\
    \KR_5(q):=\sum\limits_{m,n\geq 0} \frac{q^{m^2+3mn+3n^2+2m+4n}(1+q+q^{m+3n+2})}{(q;q)_m (q^3;q^3)_n}\stackrel{?}{=}\frac{1}{(q,q^4,q^6,q^7;q^9)_{\infty}}. \label{KR5}
\end{gather}}%
The first four double-sums on the right-hand sides are due to Kur\c{s}ung\"oz \cite{Ku}.    As is pointed out in \cite{UZ}, the fifth conjecture's combinatorial interpretation is found in \cite{M1} and the double-sum form can be constructed using Kur\c{s}ung\"oz's technique.    Whereas the infinite products of the first three identities are each symmetric, the infinite products of the last two identities are dual to each other in the sense discussed in \cite{AlAn}.  

\subsection{Known results}
Uncu and Zudilin devised combinatorial methods to create finite versions of identities $\KR_{1}(q)$ through $\KR_{5}(q)$.  We point out that where Andrews's finite versions were parametrized by $N$ and he had to consider the cases $N=2M$ and $N=2M+1$, Uncu and Zudilin's finite versions were parametrized by $N$ where they considered the cases $N=3M+a$, $a\in\{ 0,1,2 \}$.  We provide a few examples from their work to motivate our approach.

The finite versions for the first four Kanade--Russell identities read \cite[(16)-(19)]{UZ}:
{\allowdisplaybreaks \begin{align*}
\KR_1(q;N)&=\sum\limits_{n,m\geq 0} q^{m^2+3mn+3n^2}\begin{bmatrix}
N-m-3n+1\\m
\end{bmatrix}_{q}\begin{bmatrix}
\lfloor\frac{2}{3}N\rfloor-m-n+1\\n
\end{bmatrix}_{q^3},\\
\KR_2(q;N)&=\sum\limits_{n,m\geq 0} q^{m^2+3mn+3n^2+m+3n}\begin{bmatrix}
N-m-3n\\m
\end{bmatrix}_{q}\begin{bmatrix}
\lfloor\frac{2}{3}N\rfloor-m-n\\n
\end{bmatrix}_{q^3},\\
\KR_3(q;N)&=\sum\limits_{n,m\geq 0} q^{m^2+3mn+3n^2+2m+3n}\begin{bmatrix}
N-m-3n-1\\m
\end{bmatrix}_{q}^{\star} \begin{bmatrix}
\lfloor\frac{2}{3}N\rfloor-m-n\\n
\end{bmatrix}_{q^3},\\
 \KR_4(q;N)&=\sum\limits_{n,m\geq 0} q^{m^2+3mn+3n^2+m+2n}\begin{bmatrix}
N-m-3n\\m
\end{bmatrix}_q \begin{bmatrix}
\lfloor\frac{2}{3}(N-1)\rfloor-m-n+1\\n
\end{bmatrix}_{q^3}.
\end{align*}}%

\begin{remark}  In \cite{UZ}, they impose the additional condition that the $q$-binomial coefficient $\left [ \begin{matrix} N-3n-m-1\\m \end{matrix}\right]_q^{\star}$ is understood to be $1$ when it becomes $\begin{bmatrix} -1\\0 \end{bmatrix}_q$, which only happens when $N\equiv 0 \pmod 3$.
\end{remark}

The finite version for $\KR_5(q)$ is more involved and reads \cite[(20)]{UZ}:
{\allowdisplaybreaks \begin{align*}
    &\KR_5(q;N)=\sum\limits_{n,m\geq 0} q^{m^2+3mn+3n^2+2m+4n}(1+q) \\
    &\qquad \qquad \qquad \qquad \times \begin{bmatrix}
N-m-3n-1\\m
\end{bmatrix}_q  \begin{bmatrix}
\lfloor\frac{2}{3}(N-2)\rfloor-m-n+1\\n
\end{bmatrix}_{q^3}\\
&\qquad+ \sum\limits_{n,m\geq 0} q^{m^2+3mn+3n^2+3m+7n+2} \begin{bmatrix}
N-m-3n-2\\m
\end{bmatrix}_q \begin{bmatrix}
\lfloor\frac{2}{3}(N-2)\rfloor-m-n+ \delta_{3\mid (N-2)}\\n
\end{bmatrix}_{q^3}.
\end{align*}}%

Uncu and Zudilin studied the $q\to q^{-1}$ reflections of the finite versions of the Kanade--Russell identities.  We recall what happens when one makes the $q\to q^{-1}$ transformation in the $q$-Pochhammer symbol and $q$-binomial coefficient.  If we take $\rho = q^{-1}$, then
\begin{equation}
(a;\rho)_{n}=(a^{-1};q)_{n}(-a)^{n}\rho^{\binom{n}{2}}.
\end{equation}
If we take $\rho=q^{-1}$, then  
\begin{equation}
\begin{bmatrix}n+m\\m\end{bmatrix}_{\rho}=q^{-nm}\begin{bmatrix}n+m\\m\end{bmatrix}_q.
\end{equation}

Let us recall Warnaar's conjectures \cite{UZ}.  If we define
\begin{align*}
    \RK_4(q,3M)&:=q^{M(3M+2)}\KR_4(1/q,3M),\\
    \RK_4(q,3M+1)&:=q^{M(3M+5)}\KR_4(1/q,3M+1),\\
    \RK_4(q,3M+2)&:=q^{(M+1)(3M+2)}\KR_4(1/q,3M+2),
\end{align*}
then 
\begin{align*}
    \RK_4(q,3\infty)&:=\lim\limits_{M\rightarrow\infty}\RK_4(q,3M)
    \stackrel{?}{=}\frac{1}{(q^2;q^3)_{\infty}(q^3,q^9,q^{12},q^{21},q^{30},q^{36},q^{39};q^{45})_{\infty}},\\
    \RK_4(q,3\infty+1)&:=\lim\limits_{M\rightarrow\infty}\RK_4(q,3M+1)
    \stackrel{?}{=}\frac{1}{(q^2;q^3)_{\infty}(q^3,q^{12},q^{18},q^{21},q^{27},q^{30},q^{39};q^{45})_{\infty}},\\
    \RK_4(q,3\infty+2)&:=\lim\limits_{M\rightarrow\infty}\RK_4(q,3M+2)
     \stackrel{?}{=} \ \RK_4(q,3\infty)+q^2\RK_4(q,3\infty+1).
\end{align*}

There is another type of form for the limits of the finite $q$-reflections.  For example, we have \cite[(28)-(29)]{UZ} read
{\allowdisplaybreaks \begin{align*}
\RK_1(q,3\infty )&=\lim\limits_{M\rightarrow\infty}q^{3M(M+1)}\KR_1(q^{-1},3M)
=\sum\limits_{m,n\geq 0} \frac{q^{a^2-3ab+3b^2-1}}{(q^3;q^3)_b}
\begin{bmatrix}3b-1-1\\a\end{bmatrix}_q\\
& \stackrel{?}{=}
\langle 2,8,11,20 \rangle + q^{3} \langle 2,14,20,22 \rangle - q^{8} \langle 17,19,20,22 \rangle\\
&=\langle 1,8,13,20 \rangle - q \langle 4,7,13,20 \rangle + q^{5} \langle 7,16,17,20 \rangle,
\end{align*}}%
and
\begin{align*}
\RK_1(q,3\infty + 2)&=\lim\limits_{M\rightarrow\infty}q^{3(M+1)^2}\KR_1(q^{-1},3M+2)
=\sum\limits_{m,n\geq 0} \frac{q^{a^2-3ab+3b^2}}{(q^3;q^3)_b}
\begin{bmatrix}3b-a\\a\end{bmatrix}_{q}\\
& \stackrel{?}{=}
\langle 1,7,11,20 \rangle + q^{6} \langle 11,13,14,20 \rangle - q^{6} \langle 8,14,19,20 \rangle\\
&=\langle 1,4,17,20 \rangle - q^4 \langle 2,16,19,20 \rangle - q^{5} \langle 4,16,20,22 \rangle,
\end{align*}
where 
\begin{equation*}
\langle c_1,c_2,c_3,c_4 \rangle = \frac{(q^{45};q^{45})_{\infty}}
{(q^{3};q^{3})_{\infty}\prod_{j=1}^{4}(q^{c_j};q^{45-c_j};q^{45})_{\infty}}.
\end{equation*}
We also have the complementary \cite[(35)]{UZ} 
\begin{align*}
\RK_1(q,3\infty + 1)&=\lim\limits_{M\rightarrow\infty}q^{3M(M+1)+1}\KR_1(q^{-1},3M+1)
=\sum\limits_{m,n\geq 0} \frac{q^{a^2-3ab+3b^2}}{(q^3;q^3)_b}
\begin{bmatrix}3b-a+1\\a\end{bmatrix}_{q}\\
& =q\RK_1(q,3\infty )+\RK_1(q,3\infty + 2).
\end{align*}

To better understand the Kanade--Russell double-sums, let us introduce the principal series that we want to study.  We have
\begin{equation}
S(a,b;q):=\sum\limits_{m,n\geq 0} \frac{q^{m^2+3mn+3n^2+am+bn}}{(q;q)_m (q^3;q^3)_n}.
\label{equation:S-def}
\end{equation}

\noindent Let us rewrite the Kur\c{s}ung\"oz double-sums in terms of function (\ref{equation:S-def}):
{\allowdisplaybreaks \begin{gather}
\KR_1(q)=S(0,0;q),\label{equation:SKR1}\\
\KR_2(q)=S(1,3;q),\label{equation:SKR2}\\
\KR_3(q)=S(2,3;q),\label{equation:SKR3}\\
\KR_4(q)=S(1,2;q),\label{equation:SKR4}\\
\KR_5(q)=(1+q)S(2,4;q)+q^2S(3,7;q).\label{equation:SKR5}
\end{gather}}%

We are now ready to introduce Hickerson's recent conjectures.

\section{Hickerson's (conjectural) identities}\label{section:Hickerson}

Hickerson \cite{Hick} showed that the $\C(q)$-span of $S(a,b)$ (see (\ref{equation:S-def})) is $9$-dimensional over $\mathbb{C}(q)$, with five independent elements coming from the combined work of Kanade and Russell, Kur\c{s}ung\"oz, and Uncu and Zudilin.  The remaining four were found by Hickerson.

Hickerson \cite{Hick} observed that the sums in identities (\ref{KR1})-(\ref{KR5}) can be written in terms of $S(a,b)$ and that one can obtain contiguous relations of them. We obtain the first relation as follows. Multiply both the numerator and denominator by  $1-q^{m+1}$  and  replace $m$ by $m-1$ to obtain
\begin{equation}
     S(a,b) = q^{1-a} (S(a-2,b-3) - S(a-1,b-3)).                    \label{equation:A}
\end{equation}
For the second relation, multiply the numerator and denominator by  $1-q^{3n+3}$  and  
replace $n$ by $n-1$ to obtain
\begin{equation}
     S(a,b) = q^{3-b} (S(a-3,b-6) - S(a-3,b-3)).                    \label{equation:B}
\end{equation}
Using (\ref{equation:A}) and (\ref{equation:B}) one can show  
that all values of $S(a,b)$ are linear combinations, with coefficients  
that are Laurent polynomials in $q$, of just nine specific values of $S(a,b)$, say with $0\le a\le 2$ and $0\le b\le 2$.
For example, using (\ref{equation:A}) we can simplify the right-hand side of equation of $\KR_{5}(q)$ in (\ref{equation:SKR5}); it equals
{\allowdisplaybreaks \begin{align*}
     (1+q) S(2,4) + q^2 S(3,7)&
     = (1+q) S(2,4) + q^2 q^{-2} (S(1,4) - S(2,4))\\
     &= q S(2,4) + S(1,4)\\
     &= \sum_{m,n\ge 0} \frac{q^{m^2+3mn+3n^2+m+4n} (1+q^{m+1})}{ (q;q)_m  
(q^3;q^3)_n}. 
\end{align*}}%

Hickerson \cite{Hick} then discovered four new identities with complex coefficients.  Here $\omega$ is a primitive third root of unity.   Two of the conjectural identities read
\begin{equation}
    S(1,1)-\omega qS(2,4)\stackrel{?}{=}\frac{(q^6;q^9)_{\infty}(\omega q,\omega^2 q^3;q^3)_{\infty}}{(q^2;q^3)_{\infty}}\label{equation:DH-conjecture1}
\end{equation}
and
\begin{equation}
    S(0,-1)+\omega^2 S(0,2)\stackrel{?}{=}\frac{\omega(q^3;q^9)_{\infty}(\omega^2 q^2,\omega q^3;q^3)_{\infty}}{(q;q^3)_{\infty}},\label{equation:DH-conjecture2}
\end{equation}
and the other two are obtain by conjugating (\ref{equation:DH-conjecture1}) and (\ref{equation:DH-conjecture2}). 

\begin{remark}
The referee has kindly pointed out that all four of the given products, along with the two asymmetric products of $I_4$ and $I_{4a}$ \cite{KR1, M1} can be written as 
\begin{equation*}
\frac{1}{(q,q^2,\omega a, \omega^2b;q^3)_{\infty}},
\end{equation*}
where $a$ and $b$ can be any elements from $\{ q,q^2,q^3\}$, as long as $ab\not \in \{q^3, q^6\}$.
\end{remark}

\section{List of finite versions for various $S(a,b;q,N)$}
We collect the necessary list of finite versions of $S(a,b;q,N)$.  Those already found in \cite{UZ} we indicate with a $(*)$, but we include them for the sake of completeness.
\begin{prop} \label{theorem:MainFiniteList} We have
{\allowdisplaybreaks \begin{align*}
S(0,-1;q,N)&=\sum\limits_{m,n\geq 0} q^{m^2+3mn+3n^2-n}
\begin{bmatrix}
N+1-3n-m\\m
\end{bmatrix}_q
\begin{bmatrix}
\lfloor\frac{2}{3}(N+2)\rfloor-m-n\\n
\end{bmatrix}_{q^3},\\
(*)S(0,0;q,N)&=\sum\limits_{m,n\geq 0} q^{m^2+3mn+3n^2}
\begin{bmatrix}
N+1-3n-m\\m
\end{bmatrix}_q
\begin{bmatrix}
\lfloor\frac{2}{3}N\rfloor+1-m-n\\n
\end{bmatrix}_{q^3},\\
S(1,1;q,N)&=\sum\limits_{m,n\geq 0} q^{m^2+3mn+3n^2+m+n}
\begin{bmatrix}
N-3n-m\\m
\end{bmatrix}_q
  \begin{bmatrix}
\lfloor\frac{2}{3}(N+1))\rfloor-m-n\\n
\end{bmatrix}_{q^3},\\
(*)S(1,2;q,N)&=\sum\limits_{m,n\geq 0} q^{m^2+3mn+3n^2+m+2n}
 \begin{bmatrix}
N-3n-m\\m
\end{bmatrix}_q
\begin{bmatrix}
\lfloor\frac{2}{3}(N-1)\rfloor+1-m-n\\n
\end{bmatrix}_{q^3},\\
(*)S(2,3;q,N)&=\sum\limits_{m,n\geq 0} q^{m^2+3mn+3n^2+2m+3n}
 \begin{bmatrix}
N-3n-m-1\\m
\end{bmatrix}_q^{\star}
\begin{bmatrix}
\lfloor\frac{2}{3}N\rfloor-m-n\\n
\end{bmatrix}_{q^3},\\
S(2,4;q,N)&=\sum\limits_{m,n\geq 0} q^{m^2+3mn+3n^2+2m+4n}
\begin{bmatrix}
N-3n-m-1\\m
\end{bmatrix}_q\\
&\qquad \cdot 
 \begin{bmatrix}
\lfloor\frac{2}{3}(N-1)\rfloor+\delta_{3\mid (N-2)}-m-n\\n
\end{bmatrix}_{q^3},\\
S(3,5;q,N)&=\sum\limits_{m,n\geq 0} q^{m^2+3mn+3n^2+3m+5n}
\begin{bmatrix}
N-3n-m-2\\m
\end{bmatrix}_q
 \begin{bmatrix}
\lfloor\frac{2}{3}(N-1)\rfloor-m-n\\n
\end{bmatrix}_{q^3},\\
S(3,6;q,N)&=\sum\limits_{m,n\geq 0} q^{m^2+3mn+3n^2+3m+6n}
\begin{bmatrix}
N-3n-m-2\\m
\end{bmatrix}_q
 \begin{bmatrix}
\lfloor\tfrac{2}{3}N\rfloor-1-m-n\\n
\end{bmatrix}_{q^3}.
\end{align*}}%
\end{prop}

The eight finite versions of Proposition \ref{theorem:MainFiniteList} can all be obtained by a straightforward modification of the combinatorial interpretation found in \cite{Ku, Ku2, UZ}.   One immediately notices that $S(0,0;q,N)$, $S(1,2;q,N)$, and $S(2,3;q,N)$ are the finite versions of $\KR_{1}(q)$, $\KR_{4}(q)$,  and $\KR_{3}(q)$, and that $S(0,-1;q,N)$, $S(1,1;q,N)$, and $S(2,4;q,N)$ are the finite versions of three of the functions from Hickerson's list.  Naturally, one wonders why the finite versions for $\KR_{2}(q)$ and for $S(0,2)$ from Hickerson's list are absent, and why $S(3,5;q,N)$ and $S(3,6;q,N)$ are present.  We will omit discussion of the finite version of $\KR_{5}(q)$.

The finite versions for $\KR_{2}(q)$ and for $S(0,2)$ can be obtained from Proposition \ref{theorem:MainFiniteList} by using the functional equations (\ref{equation:A}) and (\ref{equation:B}) as well as the $q$-binomial coefficient properties found in Theorem \ref{thm: prop_bin_coeff}.  This is where $S(3,5;q,N)$ and $S(3,6;q,N)$ come in.

\begin{prop} \label{theorem:SecondFiniteList} We have
\begin{align*}
(*)S(1,3;q,N)
&=\sum\limits_{m,n\geq 0} q^{m^2+3mn+3n^2+m+3n}\begin{bmatrix}
N-m-3n\\m
\end{bmatrix}_q \begin{bmatrix}
\lfloor\frac{2}{3}N\rfloor-m-n\\n
\end{bmatrix}_{q^3},\\
S(0,2;q,N)
&=\sum\limits_{m,n\geq 0} q^{m^2+3mn+3n^2+2n}\begin{bmatrix}
N+1-m-3n\\m
\end{bmatrix}_q \begin{bmatrix}
\lfloor\frac{2}{3}(N-1)\rfloor+1-m-n\\ n
\end{bmatrix}_{q^3}.
\end{align*}
\end{prop}

\begin{proof}[ Proof of Proposition \ref{theorem:MainFiniteList}]
We will provide a detailed explanation of the algorithm for $S(0,-1)$, and similar formulas can be derived in the same manner. We have adapted the methods used in \cite{Ku, Ku2, UZ} to suit our specific situation.  We have also identified some typographical errors in \cite{UZ}.  

Our goal is to demonstrate that $S(0,-1;q,N)$ serves as the generating function for the number of partitions into parts less than or equal to $N$ that satisfy certain gap conditions. These conditions are analogous to \cite[(a)-(d)]{UZ} and read as follows:
\begin{enumerate}
\item[(a)]  The largest part of the partition is at most $N$.
\item[(b)]  The smallest possible value for each part is $1$.
\item[(c)]  The difference between parts that are two positions apart is at least $3$.
\item[(d)]  If consecutive parts differ by at most $1$, their sum is congruent to $2$ modulo $3$.
\end{enumerate}
It is important to note that conditions (b) and (d) may change for other versions of this algorithm.  

We will now explain the transformation of the summand:

\begin{equation}
\frac{q^{m^2+3mn+3n^2-n}}{(q;q)_m(q^3;q^3)_n}.\label{equation:inf-term}
\end{equation}

First, let us focus on interpreting and transforming:

\begin{equation*}
\frac{q^{m^2+3mn+3n^2-n}}{(q;q)_m}.
\end{equation*}
Consider the partition of $m^2+3mn+3n^2-n$ into parts:

\begin{equation*}
    \pi_{n,m}=(\underline{1,1},\underline{4,4},...,\underline{3n-2,3n-2},3n+1,3n+3,...,3n+2m-1).
\end{equation*}
Following Kur\c{s}ung\"oz \cite{Ku, Ku2}, the underlined consecutive parts of $\pi_{n,m}$ are referred to as "pairs," while the remaining terms are known as "singletons." It is worth noting that $\pi_{n,m}$ represents the minimal configuration with $2n+m$ parts and $n$ minimal gaps.  We also point out that for the congruence condition in $(d)$, there are two types of pairs satisfying the congruence condition modulo $3$, but only one of which appears in the minimal partition.  We now explain how to agitate the above minimal partition into all other partitions satisfing the gap conditions $(a)$-$(d)$.

First, we focus on the singletons, i.e. the parts that are not underlined.  We have the freedom to add any non-negative integer value $r_m$ to the largest part $3n+2m-1$, resulting in a distinct partition, where we view the singletons as a separate partition.   Similarly, after adding $r_m$ to the largest part, we can add a non-negative integer $r_{m-1}$ (where $r_{m-1} \leq r_m$) to the second largest value in $\pi_{n,m}$. By repeating this process, we can add $(r_1,\dots,r_m)$, where $0\leq r_1\leq \dots\leq r_m$. The gap conditions continue to be satisfied.   The generating function for such $(r_1, \dots, r_m)$ is given by $(q;q)_m^{-1}$.

In the finite version of this algorithm, $r_m$ must be less than or equal to $N-(3n+2m-1)$. Therefore, the generating function for partitions into at most $m$ parts is replaced by the generating function for partitions into at most $m$ parts, with each part being less than or equal to $N-(3n+2m-1)$. In a sense, we replace $(q;q)_{m}^{-1}$ in (\ref{equation:inf-term})

\begin{equation*}
     \begin{bmatrix}
N-(3n+2m-1)+m\\m
\end{bmatrix}_q.
\end{equation*}

Thus, we can conclude that 

\begin{equation*}
   q^{m^2+3mn+3n^2-n} \begin{bmatrix}
N-(3n+2m-1)+m\\m
\end{bmatrix}_q
\end{equation*}
serves as the generating function for partitions of the form:

\begin{equation*}
    \pi_{n}=(\underline{1,1},\underline{4,4},...,\underline{3n-2,3n-2},s_1,s_2,\dots,s_m)
\end{equation*}
with the following conditions:

\begin{equation*}
     s_i-s_{i-1} \geq 2 \;\text{for}\; i=2\dots m, s_1\ge 3n+1, s_m\le N.
\end{equation*}

Having completed our discussion of singletons, we focus on the pairs, i.e. the parts that are underlined.  This part of the algorithm is more difficult.  When agitating singletons, we only need to keep in mind an upper bound for the largest singleton and then stopping short of the next singleton.  For pairs we proceed in a similar fashion.  However, can also encounter a second type of pair, and we have to worry about crossing singletons.

Let us take a partition $\pi_n$ with $m$ singletons.  Let us focus on a specific pair.   We assume that there are no singletons between this pair and the next largest pair.  In this situation, there are two possible unobstructed forward motions for pairs:

\begin{equation*}
    \underline{3k-2,3k-2} \rightarrow \underline{3k-1,3k};\quad \underline{3k-1,3k} \rightarrow \underline{3k+1,3k+1}.
\end{equation*}
We note that pairs of the form $\underline{3k-1,3k}$ are new; indeed, they do not appear in the minimal configuration.

In each of the two forward motions, the total change in size of the partition is $3$.   When we moved singletons, we first moved the largest singleton, then we moved the next largest singleton, etc.   By doing so, we did not have to worry about singletons crossing each other.  For pairs we operate similarly.   However, unlike singletons, when we move a pair, we may encounter a singleton. In order to account for this situation, we define two reversible crossing-over rules for pairs over singletons:
\begin{equation*}
    \underline{3k-2,3k-2},3k+1 \rightarrow 3k-2,\underline{3k+1,3k+1},
\end{equation*}
\begin{equation*}
    \underline{3k-1,3k},3k+2 \rightarrow 3k-1,\underline{3k+2,3k+3}.
\end{equation*}
Notice that these crossing-over rules also add $3$ to the norm of the partition. The forward motion lists of the pairs relates to partitions into $\leq n$ parts, like when agitating the singletons.   However, each forward motion of a pair adds $3$ to the norm of the partition, so we instead use $(q^3;q^3)_n^{-1}$.  In \cite{UZ}, this is factor is incorrectly stated as $(q^3;q^3)_{\infty}^{-1}$.

For a finite analog, we need to adjust the reversible forward motion of the pairs. When moving a given pair, the middle point of the pair (the arithmetic mean of the elements in the pair) moves $3/2$ steps.   In \cite{UZ}, this is incorrectly stated as $2/3$ steps.  If we consider the distance from the largest part $\underline{3n-2,3n-2}$ to the upper bound $N$, there are exactly $$\lfloor\frac{2}{3}(N-(3n-2))\rfloor$$ steps.  At this point we need to be very careful, because we must keep in mind that as the pair moves forward, it can cross up to $m$ singletons.  The act of crossing over a singleton means cuts down on the number of possible steps.  The act of crossing over a singleton means that one misses a possible location where the pair can stop.  Hence, the actual number of steps a pair can move forward is $\lfloor\frac{2}{3}(N-(3n-2))\rfloor - m$.  Therefore, we need to replace the generating function $(q^3;q^3)_n^{-1}$ in (\ref{equation:inf-term}) with the $q$-binomial coefficient

\begin{equation*}
    \begin{bmatrix}
\lfloor\frac{2}{3}(N-(3n-2))\rfloor-m+n\\n
\end{bmatrix}_{q^3}
=
    \begin{bmatrix}
\lfloor\frac{2}{3}(N+2)\rfloor-m-n\\n
\end{bmatrix}_{q^3}.
\end{equation*}

The above discussion leads to the finite analog
\begin{equation*}
S(0,-1;q,N):=\sum\limits_{m,n\geq 0} q^{m^2+3mn+3n^2-n}
\begin{bmatrix}
N+1-3n-m\\m
\end{bmatrix}_q
\begin{bmatrix}
\lfloor\frac{2}{3}(N+2)\rfloor-m-n\\n
\end{bmatrix}_{q^3}.
\end{equation*}

The derivations for the other seven finite versions are entirely analogous.  For reference, we present the list of remaining seven minimal configurations:
\begin{itemize}
\item For $S(0,0)$, we consider $m^2+3mn+3n^2$ with the minimal configuration
\begin{equation*}
\pi_{n,m}=(\underline{1,2},\underline{4,5},...,\underline{3n-2,3n-1},3n+1,3n+3,...,3n+2m-1).
\end{equation*}    
\item For $S(1,1)$, we consider $m^2+3mn+3n^2+m+n$ with the minimal configuration
\begin{equation*}
\pi_{n,m}=(\underline{2,2},\underline{5,5},...,\underline{3n-1,3n-1},3n+2,3n+4,...,3n+2m).
\end{equation*} 
\item For $S(1,2)$, we consider $m^2+3mn+3n^2+m+2n$ with the minimal configuration
\begin{equation*}
\pi_{n,m}=(\underline{2,3},\underline{5,6},...,\underline{3n-1,3n},3n+2,3n+4,...,3n+2m).
\end{equation*} 
\item For $S(2,3)$, we consider $m^2+3mn+3n^2+2m+3n$ with the minimal configuration
\begin{equation*}
\pi_{n,m}=(\underline{3,3},\underline{6,6},...,\underline{3n,3n},3n+3,3n+5,...,3n+2m+1).
\end{equation*} 
\item For $S(2,4)$, we consider $m^2+3mn+3n^2+2m+4n$ with the minimal configuration
\begin{equation*}
\pi_{n,m}=(\underline{3,4},\underline{6,7},...,\underline{3n,3n+1},3n+3,3n+5,...,3n+2m+1).
\end{equation*} 
\item For $S(3,5)$, we consider $m^2+3mn+3n^2+3m+5n$ with the minimal configuration
\begin{equation*}
\pi_{n,m}=(\underline{4,4},\underline{7,7},...,\underline{3n+1,3n+1},3n+4,3n+6,...,3n+2m+2).
\end{equation*}   
\item For $S(3,6)$, we consider $m^2+3mn+3n^2+3m+6n$ with the minimal configuration
\begin{equation*}
\pi_{n,m}=(\underline{4,5},\underline{7,8},...,\underline{3n+1,3n+2},3n+4,3n+6,...,3n+2m+2).\qedhere
\end{equation*}               
\end{itemize}
\end{proof}

\begin{proof}[Proof of Proposition \ref{theorem:SecondFiniteList}]  The finite version of $S(1,3)$ can be obtained as follows.  We recall (\ref{equation:A}) in the form
\begin{equation*}
S(1,3)=S(2,3)+q^2S(3,6),
\end{equation*}
and then combine the finite versions of $S(2,3)$ and $S(3,6)$ using Theorem \ref{thm: prop_bin_coeff} (4).  The finite version of $S(0,2)$ can be obtained as follows.  We recall (\ref{equation:B}) in the form
\begin{equation*}
S(0,2)=S(0,-1)-q^2S(3,5),
\end{equation*}
and then combine the finite versions of $S(0,-1)$ and $S(3,5)$ using Theorem \ref{thm: prop_bin_coeff} (4). 
\end{proof}

\section{List of limiting cases of $q$-reflections}
In this section we will find limiting relations for the reflections of $S(1,1;q)$, $S(2,4;q)$, $S(0,-1;q)$, $S(0,2;q)$.

\begin{thm}\label{theorem:S11-limits}
We have the following limit functions for $S(1,1;q)$:
{\allowdisplaybreaks \begin{align*}
F_{1}(q):=\lim_{M\to \infty}S(1,1;q^{-1},3M+1)q^{(M+1)(3M+1)}&=\sum_{a,b\ge 0}\frac{q^{a^2-3ab+3b^2+a-b}}{(q^3;q^3)_{b}}
  \begin{bmatrix} 3b-a-1\\ a\end{bmatrix}_q,\\
F_{0}(q):=\lim_{M\to \infty}S(1,1;q^{-1},3M)q^{M(3M+1)}&=\sum_{a,b\ge 0}\frac{q^{a^2-3ab+3b^2+a-b}}{(q^3;q^3)_{b}}
  \begin{bmatrix} 3b-a\\ a\end{bmatrix}_q,\\
F_{2}(q):=\lim_{M\to \infty}S(1,1;q^{-1};3M+2)q^{(M+2)(3M+1)}
&=\sum_{a,b\ge 0}\frac{q^{a^2-3ab+3b^2+a-b-2}}{(q^3;q^3)_{b}}
  \begin{bmatrix} 3b-a-2\\ a\end{bmatrix}_q.
\end{align*}}%
\end{thm}

\begin{proof}[Proof of Theorem \ref{theorem:S11-limits}] 
We make the reflection $q\to q^{-1}$ in the finite version of $S(1,1;q)$ found in Proposition \ref{theorem:MainFiniteList}.  For the case $N=3M+1$, we can  write
{\allowdisplaybreaks \begin{align*}
S(1,1;q^{-1},3M+1)&=\sum_{n,m\ge 0}q^{m^2+3mn+3n^2-3mM-6nM-2m-4n}\\
& \qquad \cdot  \begin{bmatrix} 3M+2-m-3n\\ m\end{bmatrix}_q
 \begin{bmatrix} 2M+1 -m-n\\ n\end{bmatrix}_{q^3}.
\end{align*}}%
Changing the summation to one over $a=3M+1-2m-3n$, $b=2M+1-m-2n$ (or equivalently $m=3b-2a-1$, $n=a-2b+M+1$).  We note that the equivalence is easily seen from the invertible transformation
\begin{equation*}
\left( \begin{matrix} a\\b \end{matrix}\right)
= \left( \begin{matrix} 3M+1\\2M+1\end{matrix}\right)-\left(\begin{matrix} 2&3\\1&2\end{matrix} \right)
\left( \begin{matrix} m\\n\end{matrix}\right).
\end{equation*}
The sum transforms to
\begin{equation*}
S(1,1;q^{-1},3M+1)=\sum_{a,b\ge0}q^{a^2-3ab+3b^2+a-b-3M^2-4M-1}
  \begin{bmatrix} 3b-a-1\\ a\end{bmatrix}_q
 \begin{bmatrix} M-b+a+1\\ b\end{bmatrix}_{q^3},
\end{equation*}
where
\begin{equation*}
\lim_{M\to \infty}S(1,1;q^{-1},3M+1)q^{(M+1)(3M+1)}=\sum_{a,b\ge 0 }\frac{q^{a^2-3ab+3b^2+a-b}}{(q^3;q^3)_{b}}
 \begin{bmatrix} 3b-a-1\\ a\end{bmatrix}_q.
\end{equation*}

For the case $N=3M$, we can then obtain
\begin{align*}
S(1,1;q^{-1},3M)&=\sum_{n,m\ge 0}q^{m^2+3mn+3n^2-3mM-6nM-m-n}\\
& \ \ \ \ \ \cdot  \begin{bmatrix} 3M-m-3n\\ m\end{bmatrix}_q
 \begin{bmatrix} 2M -m-n\\ n\end{bmatrix}_{q^3}.
\end{align*}
Changing the summation to one over $a=3M-2m-3n$, $b=2M-m-2n$ (or equivalently $m=3b-2a$, $n=M+a-2b$), the sum transforms to
\begin{equation*}
S(1,1;q^{-1},3M)=\sum_{a,b\ge 0}q^{a^2-3ab+3b^2+a-b-3M^2-M}
  \begin{bmatrix} 3b-a\\ a\end{bmatrix}_q
\begin{bmatrix} M-b+a\\ b\end{bmatrix}_{q^3},
\end{equation*}
where
\begin{equation*}
\lim_{M\to \infty}S(1,1;q^{-1},3M)q^{M(3M+1)}=\sum_{a,b\ge 0}\frac{q^{a^2-3ab+3b^2+a-b}}{(q^3;q^3)_{b}}
  \begin{bmatrix} 3b-a\\ a\end{bmatrix}_q.
\end{equation*}

For the case $N=3M+2$, we then have
\begin{align*}
S(1,1;q^{-1},3M+2)&=\sum_{n,m\ge 0}q^{m^2+3mn+3n^2-3mM-6nM-3m-7n}\\
& \ \ \ \ \ \cdot \begin{bmatrix} 3M+2-m-3n\\ m\end{bmatrix}_q
 \begin{bmatrix} 2M+2 -m-n\\ n\end{bmatrix}_{q^3}.
\end{align*}
Changing the summation to one over $a=3M+2-2m-3n$, $b=2M+2-m-2n$ (or equivalently $m=3b-2a-2$, $n=M+2+a-2b$), the sum transforms to
{\allowdisplaybreaks \begin{align*}
S&(1,1;q^{-1},3M+2)\\
&=\sum_{a,b\ge 0}q^{a^2-3ab+3b^2+a-b-3M^2-7M-4}
  \begin{bmatrix} 3b-a-2\\ a\end{bmatrix}_q
 \begin{bmatrix} M+2-b+a\\ b\end{bmatrix}_{q^3},
\end{align*}}%
where
\begin{equation*}
\lim_{M\to \infty}S(1,1;q^{-1},3M+2)q^{(M+2)(3M+1)}
=\sum_{a,b\ge 0}\frac{q^{a^2-3ab+3b^2+a-b-2}}{(q^3;q^3)_{b}}
  \begin{bmatrix} 3b-a-2\\ a\end{bmatrix}_q.\qedhere
\end{equation*}

\end{proof} 

\begin{thm}\label{theorem:S24-limits}
We have the following limit functions for $S(2,4;q)$:
{\allowdisplaybreaks \begin{align*}
G_{1}(q):=\lim_{M\to \infty}S(2,4;q^{-1},3M+1)q^{(M+1)(3M+1)}&=\sum_{a,b\ge 0}\frac{q^{a^2-3ab+3b^2+2b+1}}{(q^3;q^3)_{b}}
  \begin{bmatrix} 3b-a\\ a\end{bmatrix}_q,\\
G_{0}(q):=\lim_{M\to \infty}S(2,4;q^{-1},3M)q^{M(3M+1)}&=\sum_{a,b\ge 0}\frac{q^{a^2-3ab+3b^2+2b+1}}{(q^3;q^3)_{b}}
  \begin{bmatrix} 3b-a+1\\ a\end{bmatrix}_q,\\
G_{2}(q):=\lim_{M\to \infty}S(2,4;q^{-1},3M+2)q^{(M+2)(3M+1)}
&=\sum_{a,b\ge 0}\frac{q^{a^2-3ab+3b^2+2b-1}}{(q^3;q^3)_{b}}
  \begin{bmatrix} 3b-a-1\\ a\end{bmatrix}_q.
\end{align*}}%
\end{thm}
\begin{proof}[Proof of Theorem \ref{theorem:S24-limits}]
We make the reflection $q\to q^{-1}$ in the finite version of $S(2,4;q)$ found in Proposition \ref{theorem:MainFiniteList}.  For the case $N=3M+1$, we can then write
{\allowdisplaybreaks \begin{align*}
S(2,4;q^{-1},3M+1)&=\sum_{n,m\ge 0}q^{m^2+3mn+3n^2-2m-4n-3M(m+2n)}\\
& \ \ \ \ \ \cdot  \begin{bmatrix} 3M-m-3n\\ m\end{bmatrix}_q
 \begin{bmatrix} 2M -m-n\\ n\end{bmatrix}_{q^3}.\notag
\end{align*}}%
Changing the summation to one over $a=3M-2m-3n$, $b=2M-m-2n$ (or equivalently $m=3b-2a$, $n=M+a-2b$), the sum transforms to
\begin{equation*}
S(2,4;q^{-1},3M+1)=\sum_{a,b\ge 0}q^{a^2-3ab+3b^2+2b-3M^2-4M}
 \begin{bmatrix} 3b-a\\ a\end{bmatrix}_q
 \begin{bmatrix} M-b+a\\ b\end{bmatrix}_{q^3},
\end{equation*}
where
\begin{equation*}
\lim_{M\to \infty}S(2,4;q^{-1},3M+1)q^{(M+1)(3M+1)}=\sum_{a,b\ge 0}\frac{q^{a^2-3ab+3b^2+2b+1}}{(q^3;q^3)_{b}}
  \begin{bmatrix} 3b-a\\ a\end{bmatrix}_q.
\end{equation*}

For the case $N=3M$, we then obtain
{\allowdisplaybreaks \begin{align*}
S(2,4;q^{-1},3M)&=\sum_{n,m\ge 0}q^{m^2+3mn+3n^2-m-n-3mM-6nM}\\
& \ \ \ \ \ \cdot \begin{bmatrix} 3M-m-3n-1\\ m\end{bmatrix}_q
\begin{bmatrix} 2M-1 -m-n\\ n\end{bmatrix}_{q^3}.
\end{align*}}%
Changing the summation to one over $a=3M-2m-3n-1$, $b=2M-m-2n-1$ (or equivalently $m=3b-2a+1$, $n=a-2b+M-1$), the sum transforms to
\begin{equation*}
S(2,4;q^{-1},3M)=\sum_{a,b\ge 0}q^{a^2-3ab+3b^2+2b+1-3M^2-M}
  \begin{bmatrix} 3b-a+1\\ a\end{bmatrix}_q
 \begin{bmatrix} M-b+a-1\\ b\end{bmatrix}_{q^3},
\end{equation*}
where
\begin{equation*}
\lim_{M\to \infty}S(2,4;q^{-1},3M)q^{M(3M+1)}=\sum_{a,b\ge 0}\frac{q^{a^2-3ab+3b^2+2b+1}}{(q^3;q^3)_{b}}
  \begin{bmatrix} 3b-a+1\\ a\end{bmatrix}_q.
\end{equation*}

For the case $N=3M+2$,  we then have
\begin{align*}
S(2,4;q^{-1},3M+2)&=\sum_{n,m\ge 0}q^{m^2+3mn+3n^2-m-4n-m(3M+2)-3n 2M}\\
& \ \ \ \ \ \cdot  \begin{bmatrix} 3M+1-m-3n\\ m\end{bmatrix}_q
\begin{bmatrix} 2M +1-m-n\\ n\end{bmatrix}_{q^3}.
\end{align*}
Changing the summation to one over $a=3M+1-2m-3n$, $b=2M+1-m-2n$ (or equivalently $m=3b-2a-1$, $n=M+1+a-2b$), the sum transforms to
\begin{align*}
S&(2,4;q^{-1},3M+2)\\
&=\sum_{a,b\ge 0}q^{a^2-3ab+3b^2+2b-3M^2-7M-3}
 \begin{bmatrix} 3b-a-1\\ a\end{bmatrix}_q
\begin{bmatrix} M-b+a+1\\ b\end{bmatrix}_{q^3},
\end{align*}
where
\begin{equation*}
\lim_{M\to \infty}S(2,4;q^{-1},3M+2)q^{(M+2)(3M+1)}
=\sum_{a,b\ge 0}\frac{q^{a^2-3ab+3b^2+2b-1}}{(q^3;q^3)_{b}}
  \begin{bmatrix} 3b-a-1\\ a\end{bmatrix}_q.\qedhere
\end{equation*}

\end{proof}

\begin{thm}\label{theorem:S0-1-limits}
We have the following limit functions for $S(0,-1;q)$
{\allowdisplaybreaks \begin{align*}
F_{1}^{\star}(q):=\lim_{M\to \infty}S(0,-1;q^{-1},3M+1)q^{(M+1)(3M+2)}
&=\sum_{a,b\ge 0}\frac{q^{a^2-3ab+3b^2+a-2b}}{(q^3;q^3)_{b}}
  \begin{bmatrix} 3b-a-2\\ a\end{bmatrix}_q,\\
F_{0}^{\star}(q):=\lim_{M\to \infty}S(0,-1;q^{-1},3M)q^{M(3M+2)}&=\sum_{a, b\ge 0}\frac{q^{a^2-3ab+3b^2+a-2b}}{(q^3;q^3)_{b}}
  \begin{bmatrix} 3b-a-1\\ a\end{bmatrix}_q,\\
F_{2}^{\star}(q):=\lim_{M\to \infty}S(0,-1;q^{-1},3M+2)q^{(M+1)(3M+2)}
&=\sum_{a,b\ge 0}\frac{q^{a^2-3ab+3b^2+a-2b}}{(q^3;q^3)_{b}}
  \begin{bmatrix} 3b-a\\ a\end{bmatrix}_q.
\end{align*}}%
\end{thm}
\begin{proof}[Proof of Theorem \ref{theorem:S0-1-limits}]
We make the reflection $q\to q^{-1}$ in the finite version of $S(0,-1;q)$ found in Proposition \ref{theorem:MainFiniteList}.  For the case $N=3M+1$, we obtain
\begin{align*}
S(0,-1;q^{-1},3M+1)&=\sum_{n,m\ge 0}q^{m^2+3mn+3n^2-2m-5n-6nM-3mM}\\
& \ \ \ \ \ \cdot \begin{bmatrix} 3M+2-m-3n\\ m\end{bmatrix}_q
 \begin{bmatrix} 2M+2 -m-n\\ n\end{bmatrix}_{q^3}.\notag
\end{align*}
Changing the summation to one over $a=3M+2-2m-3n$, $b=2M+2-m-2n$ (or equivalently $m=3b-2a-2$, $n=M+2+a-2b$), the sum transforms to
\begin{align*}
S&(0,-1;q^{-1},3M+1)\\
&=\sum_{a,b\ge 0}q^{a^2-3ab+3b^2+a-2b-2-3M^2-5M}
  \begin{bmatrix} 3b-a-2\\ a\end{bmatrix}_q
 \begin{bmatrix} M+2-b+a\\ b\end{bmatrix}_{q^3},
\end{align*}
where
\begin{equation*}
\lim_{M\to \infty}S(0,-1;q^{-1},3M+1)q^{(M+1)(3M+2)}
=\sum_{a,b\ge 0}\frac{q^{a^2-3ab+3b^2+a-2b}}{(q^3;q^3)_{b}}
 \begin{bmatrix} 3b-a-2\\ a\end{bmatrix}_q.
\end{equation*}

For the case $N=3M$, we have
\begin{align*}
S(0,-1;q^{-1},3M)&=\sum_{n,m\ge 0}q^{m^2+3mn+3n^2+7n-6nM-3Mm}\\
& \ \ \ \ \ \cdot  \begin{bmatrix} 3M+1-m-3n\\ m\end{bmatrix}_q
 \begin{bmatrix} 2M+1-m-n\\ n\end{bmatrix}_{q^3}.
\end{align*}
Changing the summation to one over $a=3M+1-2m-3n$, $b=2M+1-m-2n$ (or equivalently $m=3b-2a-1$, $n=M+1+a-2b$), the sum transforms to
\begin{equation*}
S(0,-1;q^{-1},3M)=\sum_{a, b\ge 0}q^{a^2-3ab+3b^2+a-2b-3M^2-2M}
  \begin{bmatrix} 3b-a-1\\ a\end{bmatrix}_q
\begin{bmatrix} M-1+a-b\\ b\end{bmatrix}_{q^3},
\end{equation*}
where
\begin{equation*}
\lim_{M\to \infty}S(0,-1;q^{-1},3M)q^{M(3M+2)}=\sum_{a, b\ge 0}\frac{q^{a^2-3ab+3b^2+a-2b}}{(q^3;q^3)_{b}}
 \begin{bmatrix} 3b-a-1\\ a\end{bmatrix}_q.
\end{equation*}

For the case $N=3M+2$, we have
{\allowdisplaybreaks \begin{align*}
S(0,-1;q^{-1},3M+2)&=\sum_{n,m\ge 0}q^{m^2+3mn+3n^2-3m-5n-3mM-6nM}\\
& \qquad \cdot  \begin{bmatrix} 3M+3-m-3n\\ m\end{bmatrix}_q
 \begin{bmatrix} 2M +2-m-n\\ n\end{bmatrix}_{q^3}.
\end{align*}}%
Changing the summation to one over $a=3M+3-2m-3n$, $b=2M+2-m-2n$ (or equivalently $m=3b-2a$, $n=M+1+a-2b$), the sum transforms to
\begin{equation*}
S(0,-1;q^{-1},3M+2)
=\sum_{a,b\ge 0}q^{a^2-3ab+3b^2+a-2b-2-3M^2-5M}
  \begin{bmatrix} 3b-a\\ a\end{bmatrix}_q
 \begin{bmatrix} M+1-b+a\\ b\end{bmatrix}_{q^3},
\end{equation*}
where
\begin{equation*}
\lim_{M\to \infty}S(0,-1;q^{-1},3M+2)q^{(M+1)(3M+2)}=\sum_{a,b\ge 0}\frac{q^{a^2-3ab+3b^2+a-2b}}{(q^3;q^3)_{b}}
  \begin{bmatrix} 3b-a\\ a\end{bmatrix}_q.\qedhere
\end{equation*}
\end{proof}

\begin{thm}\label{theorem:S02-limits}
We have the following limit functions for $S(0,2;q)$
{\allowdisplaybreaks \begin{align*}
G_{1}^{\star}:=(q)\lim_{M\to \infty}S(0,2;q^{-1},3M+1)q^{(M+1)(3M+2)}
&=\sum_{a, b\ge 0}\frac{q^{a^2-3ab+3b^2-2a+4b+1}}{(q^3;q^3)_{b}}
  \begin{bmatrix} 3b-a+1\\ a\end{bmatrix}_q,\\
G_{0}^{\star}(q):=\lim_{M\to \infty}S(0,2;q^{-1},3M)q^{M(3M+2)}&=\sum_{a,b\ge 0}\frac{q^{a^2-3ab+3b^2-2a+4b+1}}{(q^3;q^3)_{b}}
  \begin{bmatrix} 3b-a+2\\ a\end{bmatrix}_q,\\
G_{2}^{\star}(q):=\lim_{M\to \infty}S(0,2;q^{-1},3M+2)q^{(M+1)(3M+2)}
&=\sum_{a,b\ge 0}\frac{q^{a^2-3ab+3b^2-2a+4b+1}}{(q^3;q^3)_{b}}
  \begin{bmatrix} 3b-a+3\\ a\end{bmatrix}_q.\notag
\end{align*}}%
\end{thm}

\begin{proof}[Proof of Theorem \ref{theorem:S02-limits}]
We make the reflection $q\to q^{-1}$ in the finite version of $S(0,2;q)$ found in Proposition \ref{theorem:SecondFiniteList}.  For the case $N=3M+1$, we then obtain
{\allowdisplaybreaks \begin{align*}
S(0,2;q^{-1},3M+1)&=\sum_{n,m\ge 0}q^{m^2+3mn+3n^2-2m-5n-3mM-6nM)}\\
& \qquad \cdot  \begin{bmatrix} 3M+2-m-3n\\ m\end{bmatrix}_q
 \begin{bmatrix} 2M+1 -m-n\\ n\end{bmatrix}_{q^3}.
\end{align*}}%
Changing the summation to one over $a=3M+2-2m-3n$, $b=2M+1-m-2n$ (or equivalently $m=1+3b-2a$, $n=M+a-2b$), the sum transforms to
\begin{equation*}
S(0,2;q^{-1},3M+1)=\sum_{a, b\ge0 }q^{a^2-3ab+3b^2-2a+4b-3M^2-5M-1}
  \begin{bmatrix} 3b-a+1\\ a\end{bmatrix}_q
 \begin{bmatrix} M-b+a\\ b\end{bmatrix}_{q^3},
\end{equation*}
where
\begin{equation*}
\lim_{M\to \infty}S(0,2;q^{-1},3M+1)q^{(M+1)(3M+2)}
=\sum_{a, b\ge 0}\frac{q^{a^2-3ab+3b^2-2a+4b+1}}{(q^3;q^3)_{b}}
  \begin{bmatrix} 3b-a+1\\ a\end{bmatrix}_q.
\end{equation*}

For the case $N=3M$, we have
{\allowdisplaybreaks \begin{align*}
S(0,2;q^{-1},3M)&=\sum_{n,m\ge 0}q^{m^2+3mn+3n^2-m-2n-3mM-6nM}\\
& \qquad \cdot  \begin{bmatrix} 3M+1-m-3n\\ m\end{bmatrix}_q
 \begin{bmatrix} 2M -m-n\\ n\end{bmatrix}_{q^3}.
\end{align*}}%
Changing the summation to one over $a=3M+1-2m-3n$, $b=2M-m-2n$ (or equivalently $m=2+3b-2a$, $n=M-1+a-2b$), the sum transforms to
\begin{equation*}
S(0,2;q^{-1},3M)=\sum_{a,b\ge 0}q^{a^2-3ab+3b^2-2a+4b+1-3M^2-2M}
  \begin{bmatrix} 3b-a+2\\ a\end{bmatrix}_q
 \begin{bmatrix} M-1-b+a\\ b\end{bmatrix}_{q^3},
\end{equation*}
where
\begin{equation*}
\lim_{M\to \infty}S(0,2;q^{-1},3M)q^{M(3M+2)}=\sum_{a,b\ge 0}\frac{q^{a^2-3ab+3b^2-2a+4b+1}}{(q^3;q^3)_{b}}
  \begin{bmatrix} 3b-a+2\\ a\end{bmatrix}_q.
\end{equation*}

For the case $N=3M+2$, we arrive at
{\allowdisplaybreaks \begin{align*}
S(0,2;q^{-1},3M+2)&=\sum_{n,m\ge 0}q^{m^2+3mn+3n^2-3m-5n-3mM-6nM}\\
& \ \ \ \ \ \cdot \begin{bmatrix} 3M+3-m-3n\\ m\end{bmatrix}_q
\begin{bmatrix} 2M+1 -m-n\\ n\end{bmatrix}_{q^3}.
\end{align*}}%
Changing the summation to one over $a=3M+3-2m-3n$, $b=2M+1-m-2n$ (or equivalently $m=3+3b-2a$, $n=M-1+a-2b$), the sum transforms to
\begin{equation*}
S(0,2;q^{-1},3M+2)=\sum_{a,b\ge 0}q^{a^2-3ab+3b^2-2a+4b-3M^2-5M-1}
  \begin{bmatrix} 3b-a+3\\ a\end{bmatrix}_q
\begin{bmatrix} M-1-b+a\\ b\end{bmatrix}_{q^3},
\end{equation*}
where
\begin{equation*}
\lim_{M\to \infty}S(0,2;q^{-1},3M+2)q^{(M+1)(3M+2)}
=\sum_{a,b\ge 0}\frac{q^{a^2-3ab+3b^2-2a+4b+1}}{(q^3;q^3)_{b}}
  \begin{bmatrix} 3b-a+3\\ a\end{bmatrix}_q.\qedhere
\end{equation*}
\end{proof}

\section{List of reflections mod $3$}
In this section we will list all the limits, and we will give a preview of identities to be proved in Section  \ref{section:IdentitiesAndConjecture}. We first list the limits for the finite versions for $S(1,1)$ and $S(2,4)$.  Here $F_{a}(q)$ is the limiting value of the reflection $S(1,1;q^{-1},3M+a)$ and $G_{a}(q)$ is the limiting value of the reflection $S(2,4;q^{-1},3M+a)$.  We have
{\allowdisplaybreaks \begin{align*}
F_{0}(q)&:=\sum_{a,b\ge 0}\frac{q^{a^2-3ab+3b^2+a-b}}{(q^3;q^3)_{b}}
  \begin{bmatrix} 3b-a\\ a\end{bmatrix}_q,\\
G_{0}(q)&:= \sum_{a,b\ge 0}\frac{q^{a^2-3ab+3b^2+2b+1}}{(q^3;q^3)_{b}}
 \begin{bmatrix} 3b-a+1\\ a\end{bmatrix}_q,\\
F_{1}(q)&:=\sum_{a,b\ge 0}\frac{q^{a^2-3ab+3b^2+a-b}}{(q^3;q^3)_{b}}
  \begin{bmatrix} 3b-a-1\\ a\end{bmatrix}_q,\\
G_{1}(q)&:= \sum_{a,b\ge 0}\frac{q^{a^2-3ab+3b^2+2b+1}}{(q^3;q^3)_{b}}
  \begin{bmatrix} 3b-a\\ a\end{bmatrix}_q,\\
F_{2}(q)&:= \sum_{a,b\ge 0}\frac{q^{a^2-3ab+3b^2+a-b-2}}{(q^3;q^3)_{b}}
  \begin{bmatrix} 3b-a-2\\ a\end{bmatrix}_q,\\
 G_{2}(q)&:=\sum_{a,b\ge 0}\frac{q^{a^2-3ab+3b^2+2b-1}}{(q^3;q^3)_{b}}
  \begin{bmatrix} 3b-a-1\\ a\end{bmatrix}_q.
\end{align*}}%
In Section \ref{section:IdentitiesAndConjecture} we will prove that
\begin{align*}
F_0(q)&=F_1(q)+q^2F_2(q)+1,\\ 
G_0(q)&=G_1(q)+q^2G_2(q). 
\end{align*}

We list the finite versions for $S(0,-1)$ and $S(0,2)$.  Here $F_{a}^{\star}(q)$ is the limiting value of the reflection $S(0,-1;q^{-1},3M+a)$ and $G_{a}^{\star}(q)$ is the limiting value of the reflection $S(0,2;q^{-1},3M+a)$.  We have
{\allowdisplaybreaks \begin{align*}
F_{0}^{\star}(q)&:=\sum_{a, b\ge 0}\frac{q^{a^2-3ab+3b^2+a-2b}}{(q^3;q^3)_{b}}
  \begin{bmatrix} 3b-a-1\\ a\end{bmatrix}_q,\\
G_{0}^{\star}(q)&:= \sum_{a,b\ge 0}\frac{q^{a^2-3ab+3b^2-2a+4b+1}}{(q^3;q^3)_{b}}
  \begin{bmatrix} 3b-a+2\\ a\end{bmatrix}_q,\\
F_{1}^{\star}(q)&:=\sum_{a,b\ge 0}\frac{q^{a^2-3ab+3b^2+a-2b}}{(q^3;q^3)_{b}}
  \begin{bmatrix} 3b-a-2\\ a\end{bmatrix}_q,\\
G_{1}^{\star}(q)&:=\sum_{a, b\ge 0}\frac{q^{a^2-3ab+3b^2-2a+4b+1}}{(q^3;q^3)_{b}}
  \begin{bmatrix} 3b-a+1\\ a\end{bmatrix}_q\\
 F_{2}^{\star}(q)&:=\sum_{a,b\ge 0}\frac{q^{a^2-3ab+3b^2+a-2b}}{(q^3;q^3)_{b}}
  \begin{bmatrix} 3b-a\\ a\end{bmatrix}_q,\\
G_{2}^{\star}(q)&:= \sum_{a,b\ge 0}\frac{q^{a^2-3ab+3b^2-2a+4b+1}}{(q^3;q^3)_{b}}
  \begin{bmatrix} 3b-a+3\\ a\end{bmatrix}_q.
\end{align*}}%

In Section \ref{section:IdentitiesAndConjecture} we will prove that
\begin{align*}
F_{2}^{\star}(q)&=F_{0}^{\star}(q)+F_{1}^{\star}(q)+1,\\
G_{2}^{\star}(q)&=G_{0}^{\star}(q)+G_{1}^{\star}(q).
\end{align*}

\section{Two conjectures and four new identities}\label{section:IdentitiesAndConjecture}
We begin this section with two conjectures which are similar to those found by Warnaar, Uncu and Zudilin \cite[(25), (26)]{UZ}:

\begin{conj} \label{conj:Conjecture-1}
The following identity is expected to be true:
\begin{equation*}
F_1(q)- \omega q^{-1} G_1(q)=-\omega\cdot 
\frac{(q^{15};q^{45})_{\infty}(\omega q^3;q^3)_{\infty}(\omega^2 q,\omega^2 q^{2},\omega^2 q^{4},\omega^2 q^{8},\omega^2 q^{10};q^{15})_{\infty}}
{(q^5,q^{11},q^{14};q^{15})_{\infty}(q^{3},q^{12},q^{18},q^{27};q^{45})_{\infty}}.
\end{equation*}
\end{conj}

\begin{conj} \label{conj:Conjecture-2}
The following identity is expected to be true:
\begin{equation*}
F_0^{\star}(q)+\omega^2 G_0^{\star}(q)=-\omega\cdot 
\frac{(q^{30};q^{45})_{\infty}(\omega^2 q^3;q^3)_{\infty}(\omega q^{5},\omega q^{7},\omega q^{11},\omega q^{13},\omega q^{14};q^{15})_{\infty}}
{(q,q^{4},q^{10};q^{15})_{\infty}(q^{18},q^{27},q^{33},q^{42};q^{45})_{\infty}}.
\end{equation*}
\end{conj}
\begin{remark}
We point out that by taking the conjugates of the two conjectures, we can obtain two new conjectures, so in a sense we have four conjectures.
\end{remark}

Other combinations do not appear to yield nice quotients and are perhaps more like conjectural identities (29)-(32) of \cite{UZ}.

We now prove identities similar to those found by Uncu and Zudilin \cite[(23), (27)]{UZ}.

\begin{thm}\label{stats}
The following identities are true
{\allowdisplaybreaks \begin{align}
F_0(q)&=F_1(q)+q^2F_2(q)+1,\label{equation:DR-finiteconjecture-S11S24-blockF}\\
G_0(q)&=G_1(q)+q^2G_2(q),\label{equation:DR-finiteconjecture-S11S24-blockG}\\
F_{2}^{\star}(q)&=F_{0}^{\star}(q)+F_{1}^{\star}(q)+1,\label{equation:DR-finiteconjecture-S11S24-blockF_star}\\
G_{2}^{\star}(q)&=G_{0}^{\star}(q)+G_{1}^{\star}(q).\label{equation:DR-finiteconjecture-S11S24-blockG_star}
\end{align}}%
\end{thm}

\begin{proof}
The proofs for (\ref{equation:DR-finiteconjecture-S11S24-blockF}), (\ref{equation:DR-finiteconjecture-S11S24-blockG}), (\ref{equation:DR-finiteconjecture-S11S24-blockF_star}), and (\ref{equation:DR-finiteconjecture-S11S24-blockG_star}) are all similar, so we will only prove the first identity.  We begin by writing
{\allowdisplaybreaks \begin{align*}
q^2F_2(q)&+F_{1}(q)-F_0(q)\\
&=\sum\limits_{b\ge0}\frac{q^{3b^2-b}}{(q^3;q^3)_{b}}\left [ \sum\limits_{a\ge0} q^{a(a-3b+1)}\left( \begin{bmatrix} 3b-a-2\\ a\end{bmatrix}_q+ \begin{bmatrix} 3b-a-1\\ a\end{bmatrix}_q- \begin{bmatrix} 3b-a\\ a\end{bmatrix}_q\right)\right ].
\end{align*}}%
We show that the term in brackets is equal to zero.  The case $b=0$ is straightforward and is where the ``$+1$'' comes from.  Note that if we add the condition that
\begin{equation*}
 \begin{bmatrix} -1\\0 \end{bmatrix}_{q}:=1,
\end{equation*}
then we would have $F_0(q)=F_1(q)+q^2F_2(q)$.  Identity (\ref{equation:DR-finiteconjecture-S11S24-blockF_star}) behaves in a similar fashion, but we would pick up a ``$+1$'' in identity (\ref{equation:DR-finiteconjecture-S11S24-blockG}).

For $b\ge 1$ we use Theorem \ref{thm: prop_bin_coeff} (3) to write
{\allowdisplaybreaks \begin{align*}
\sum\limits_{a\ge0} &q^{a(a-3b+1)}\left( \begin{bmatrix} 3b-a-2\\ a\end{bmatrix}_q+ \begin{bmatrix} 3b-a-1\\ a\end{bmatrix}_q- \begin{bmatrix} 3b-a\\ a\end{bmatrix}_q\right)\\
&=\sum\limits_{a\ge0} q^{a(a-3b+1)}\left( \begin{bmatrix} 3b-a-2\\ a\end{bmatrix}_q
+ \begin{bmatrix} 3b-a-1\\ a\end{bmatrix}_q\right)\\
&\qquad - \sum\limits_{a\ge0} q^{a(a-3b+1)}\left( \begin{bmatrix} 3b-a-1\\ a\end{bmatrix}_q
+q^{3b-2a}\begin{bmatrix} 3b-a-1\\ a-1\end{bmatrix}_q\right )\\
&=\sum\limits_{a\ge0} q^{a(a-3b+1)}\left( \begin{bmatrix} 3b-a-2\\ a\end{bmatrix}_q
 -q^{3b-2a}\begin{bmatrix} 3b-a-1\\ a-1\end{bmatrix}_q\right).
\end{align*}}%
We then distribute and make the substitution $a \to a+1$ in the second summand.  Thus
{\allowdisplaybreaks \begin{align*}
\sum\limits_{a\ge0} &q^{a(a-3b+1)}\left( \begin{bmatrix} 3b-a-2\\ a\end{bmatrix}_q
 -q^{3b-2a}\begin{bmatrix} 3b-a-1\\ a-1\end{bmatrix}_q\right)\\
&=\sum\limits_{a\ge0} q^{a(a-3b+1)} \begin{bmatrix} 3b-a-2\\ a\end{bmatrix}_q
 -\sum\limits_{a\ge0} q^{a(a-3b+1)} q^{3b-2a}\begin{bmatrix} 3b-a-1\\ a-1\end{bmatrix}_q
 =0.\qedhere
\end{align*}}%
\end{proof}

 \section*{Acknowledgements}
 This work was supported by the Theoretical Physics and Mathematics Advancement Foundation BASIS, agreement No. 20-7-1-25-1.  I also want to express my gratitude to Matthew Russell, Ali Uncu, and Wadim Zudilin, who gave useful recommendations and to my supervisor Eric Mortenson for an interesting topic and help in writing the work.  We would also like to thank the referee for helpful comments, that allowed us to improve the manuscript.


\begin{thebibliography}{10}
\bibitem{AlAn} K. Alladi, G. E. Andrews, {\em The dual of G\"ollnitz's (big) partition theorem}, Ramanujan J. {\bf 36} (2015), 171--201.

\bibitem{AND1}
G. E. Andrews, {\em A polynomial identity which implies the Rogers--Ramanujan identities}, Scripta Math. 28 (1970), 297--305.


\bibitem{AND3} G. E. Andrews, {\em The theory of partitions}, Reprint of the 1976 original. Cambridge Mathematical Library. Cambridge University Press, Cambridge, 1998. xvi+255 pp. 

\bibitem{AND2}
G. E. Andrews, {\em The hard-hexagon model and Rogers--Ramanujan type identities}, Proc. Natl.
Acad. Sci. USA 78 (1981), 5290--5292.

\bibitem{Baxter}
R. J. Baxter, {\em Hard hexagons: exact solution}, J. Phys. A 13 (1980), 161--170.

\bibitem{KR1} S. Kanade, M. C. Russell, {\em IdentityFinder and some new identities of Rogers--Ramanujan type}, Exp. Math. 24 (2015), no. 4, 419--423.

\bibitem{Hick} D. Hickerson, personal communication.

\bibitem{Ku} K. Kur\c{s}ung\"oz, {\em Andrew--Gordon type series for Kanade--Russell conjectures}, Ann. Combin.
{\bf 23} (2019), no. 3-4, 835--888.

\bibitem{Ku2} K. Kur\c{s}ung\"oz, {\em Andrew--Gordon type series for Capparelli's and Go\"nitz--Gordon identities}, J. Combin. Theory Ser. A {\bf 165} (2019), 117--138.

\bibitem{Mort}
E. T. Mortenson, {\em On the dual nature of partial theta functions and Appell--Lerch sums}, Adv. Math. 264 (2014), 236--260.

\bibitem{Raman} 
S. Ramanujan, {\em The Lost Notebook and Other Unpublished Papers}, Narosa Pub. House, New Delhi, 1988.

\bibitem{PSW} M. Penn, C. Sadowski, G. Webb, {\em Principal subspaces of twisted modules for certain lattice vertex operator algebras}, Internat. J. Math. {\bf 30} (2019), no. 10, 1950048, 47 pp. 

\bibitem{Rog} L. J. Rogers, {\em Second memoir on the expansion of certain infinite products}, Proc. London. Math. Soc. {\bf 25} (1894), 318--343.

\bibitem{M1} M. C. Russell, {\em Using experimental mathematics to conjecture and prove theorems in the theory of partitions and commutative and non-commutative recurrences}, PhD thesis (Rutgers University, New Brunswick, NJ, 2016).

\bibitem{Schur}
I. J. Schur, {\em Ein Beitrag zur additiven Zahlentheorie und zur Theorie der Kettenbr\"ucke}, Preuss. Akad. Wiss. Phys.-Math. Kl. (1917), 315--336.

\bibitem{Slat} L. J. Slater, {\em Further identities of Rogers--Ramanujan type}, Proc. London. Math. Soc. {\bf 54} (1952), 147--167.

\bibitem{Tsuch} S. Tsuchioka, {\em A vertex operator reformulation of the Kanade--Russell conjecture modulo $9$}, arXiv:2211.12351.

\bibitem{UZ}
A. Uncu, W. Zudilin, {\em Reflecting (on) the modulo 9 Kanade--Russell (conjectural) identities}, Seminaire Lotharingien de Combinatoire 85 (2021), Article B85e, 2021, https://arxiv.org/abs/2106.02959.
\end{thebibliography}
\end{document}